\documentclass{amsart}
\pdfoutput=1
\usepackage[utf8]{inputenc} 
\usepackage{tikz-cd}

\usepackage[T1]{fontenc}    
\usepackage{ae,aecompl}     
\usepackage{url}

\usepackage{pdfsync}

\usepackage{tensor}

\input xypic

\xyoption{all}

\xyoption{arc}

\xyoption{2cell}

\usepackage{amsmath}
\usepackage{amsfonts}
\usepackage{amssymb}
\usepackage{amsthm} %
\usepackage{mathrsfs} %
\usepackage{enumerate} 
\usepackage{verbatim}	
\usepackage{dynkin-diagrams}

\usepackage{tikz}
\usepackage{tikz-cd}
\usepackage{pgfplots}
\usetikzlibrary{calc}
\usetikzlibrary{shapes}

\input xypic
\xyoption{all}
\xyoption{arc}
\xyoption{2cell}

\usetikzlibrary{intersections}
\usetikzlibrary{patterns}
\usepackage{circuitikz}

\usepackage{caption}
\usepackage[hyperfootnotes]{hyperref}

\theoremstyle{plain} 
\newtheorem{thm}{Theorem}[section]
\newtheorem{cor}[thm]{Corollary}
\newtheorem{prop}[thm]{Proposition}

\newtheorem{conject}[thm]{Conjecture}

\theoremstyle{definition}
\newtheorem{defi}[thm]{Definition}

\usepackage{color}
\definecolor{Ccolor}{rgb}{0,0.5,0}
\definecolor{Mcolor}{rgb}{1,0,0}
\definecolor{lightgray}{rgb}{0.6,0.6,0.6}

\def\weakr{\leq_{R}}

\def\sweakr{<_{R}}

\def\weakl{\leq_{L}}

\def\sweakl{<_{L}}

\def\N{\mathbb{N}}

\def\H#1{\mathcal{H}(#1)}

\author[H.~Gimenez]{Harrison Gimenez}
\address[H.~Gimenez]{Department of Mathematics \\ B26 Hayes-Healy Building \\ University of Notre Dame, Indiana 46556, U.S.A.}
\email{hgimenez@nd.edu}


\title{On the Weak Right Order of Right-Angled Coxeter Systems}


\begin{document}


\begin{abstract}
     Let $ (W,S)$ be a Coxeter system, and let $ w\in W$. Let $ [1,w] := \{ x\in W \mid x \weakr w \} $ where $ \weakr$ denotes the weak right order of $ (W,S)$. The element $ w$ is said to have the \emph{ancestor property} if there is a unique non-trivial involution of maximal length in the set $ [1,w]$. The ancestor property was first defined by Hart and Rowley in \cite{hart2025noteinvolutionprefixescoxeter} where they conjectured that all non-identity elements in a finite Coxeter system have the ancestor property. In an arbitrary Coxeter system $(W,S)$, we show that the ancestor property holds for any non-identity fully commutative element (see \cite{stembridge1996fully} for the definition of a fully commutative element). In particular, since any element of a right-angled Coxeter system is fully commutative, we show that the ancestor property holds for all non-identity elements of a right-angled Coxeter system. Lastly, we also provide an axiomatization of right-angled Coxeter systems as reflection systems with a reflection cocycle that obeys a certain property called the \emph{meet intersection condition}.

 \end{abstract}


 \maketitle


\section{Introduction}

Let $\N$ denote the set of non-negative integers, and let $ \N_{\geq 2}$ denote the set of non-negative integers that are greater than or equal to $2$. A function $ m:S \times S \rightarrow \mathbb{N} \cup \{ \infty \}$ is called a \emph{Coxeter matrix} if it satisfies the following three properties:

\begin{enumerate}
    \item for all $ r,s\in S$, $ m(r,s) =1$ if and only if $ r=s$

    \item for all $ r,s\in S$, $ r \neq s$ implies $ m(r,s) \in \N_{\geq 2} \cup \{ \infty \}$.

    \item for all $ r,s\in S$, $ m(r,s) = m(s,r)$ 
\end{enumerate}
A Coxeter matrix $ m: S \times S \rightarrow \N \cup \{ \infty \}$ is called a \emph{right-angled Coxeter matrix} if it saisfies a stronger version of property (2):

\begin{enumerate}
    \item[($2^{\ast}$)] for all $ r,s \in S$, $ r\neq s$ implies $m(r,s) \in \{ 2, \infty  \}$ 
\end{enumerate}
A \emph{Coxeter system} is an ordered pair $ (W,S)$ such that the following hold:
\begin{enumerate}
    \item $ S \subseteq W$

    \item $ m : S \times S \rightarrow \N \cup \{ \infty \}$ is a Coxeter matrix

    \item $ W $ is a group such that $ W \cong  \langle s\in S \mid \forall r,s\in S, \ (rs)^{m(r,s)} =1 \rangle$. If $ m(r,s) = \infty$, then we omit the relation $ (rs)^{m(r,s)} = 1$ from the presentation.
\end{enumerate}
A Coxeter system $ (W,S)$ is called a \emph{right-angled Coxeter system} if the associated Coxeter matrix $ m : S \times S \rightarrow \N \cup \{ \infty \}$ is a right-angled Coxeter matrix. The texts \cite{humphreys1992reflection} and \cite{bjorner2005combinatorics} provide a good introduction to the general theory of Coxeter systems.

Let $ (W,S)$ be a Coxeter system, and let $ \weakr $ denote its weak right order. If $ w\in W$, we define $ [1,w] : = \{ x\in W \mid x \weakr w \}$. The element $ w\in W$ is said to have the \emph{ancestor property} if there is a unique non-trivial involution of maximal length in the set $ [1,w]$. Hart and Rowley introduced the following conjecture \cite{hart2025noteinvolutionprefixescoxeter}:

\begin{conject}{(Hart, Rowley)}
    Let $ (W,S)$ be a finite Coxeter system ($|W|< \infty$). If $ w \in W$ and $ w\neq 1$, then $w$ has the ancestor property.
\end{conject}

Let $ (W,S)$ be a Coxeter system, and let $ w\in W$. Stembridge defined the notion of a \emph{fully commutative} element: $w$ is fully commutative if any reduced 
expression for $w$ can be obtained from any other by means of braid relations that only involve commuting generators \cite{stembridge1996fully}. We prove the following theorem that is related to the conjecture that was originally posed by Hart and Rowley:

\begin{thm} \label{solvedconj}
    Let $ (W,S)$ be a Coxeter system. Let $ w\in W$ be a fully commutative element such that $ w\neq 1$. Then $ w$ has the ancestor property. 
\end{thm}
Note that any element of a right-angled Coxeter system is fully commutative. Hence, we obtain the following corollary:

\begin{cor}
    Let $ (W,S)$ be a right-angled Coxeter system. If $ w\in W$ is such that $ w\neq 1$, then $ w$ has the ancestor property.
\end{cor}

Lastly, we provide an axiomatization of right angled Coxeter systems as reflection systems with a reflection cocycle satisfying a certain property. An ordered pair $ (G,R)$ is a \emph{reflection system} if it satisfies the following properties:
\begin{enumerate}
    \item $G$ is a group

    \item $ R$ is a set of non-identity involutions that generate $W$
\end{enumerate}
If $ (G,R)$ is a reflection system, we define $ T: = \bigcup_{g\in G}gRg^{-1} $. Given a reflection system $ (G,R)$, a \emph{reflection cocycle} is a function $ N: G \rightarrow \mathcal{P}(T)$ satisfying the following properties:

\begin{enumerate}
    \item for all $ r\in R$, $ N(r) = \{ r \}$

    \item for all $ x,y\in W$, $N(xy) = N(x) \Delta xN(y)x^{-1}$ where $ \Delta$ denotes the symmetric difference of sets 
\end{enumerate}
Dyer proved the following theorem (see Lemma 1.2 and Lemma 1.3 of \cite{dyer1987hecke}; also see the introduction of \cite{dyer1990reflection}):

\begin{thm}{(Dyer)} \label{dyerthm}
    $ (W,S)$ is a Coxeter system if and only if $ (W,S)$ is a reflection system with a reflection cocycle $ N : W \rightarrow \mathcal{P}(T)$. Furthermore, if $ (W,S)$ is a Coxeter system, then $ N: W \rightarrow \mathcal{P}(W)$ is the unique reflection cocycle associated to the reflection system $ (W,S)$, and

    $$ N(w) = \{ t\in T \mid \ell(tw) < \ell(w)   \}$$
    for all $w\in W$.
\end{thm}
We prove the following characterization of right-angled Coxeter systems:

\begin{thm} \label{gimenezracsthm}
    $ (W,S)$ is a right-angled Coxeter system if and only if $ (W,S)$ is a reflection system with a reflection cocycle $ N: W \rightarrow \mathcal{P}(T)$ satisfying the \emph{meet intersection condition}: for all $ x,y\in W$, there is a unique $ z \in W$ such that $ N(z) = N(x) \cap N(y)$.
\end{thm}

\section{Proofs}

\begin{defi}
    Let $ A$ be a set. An \emph{$A$-heap} is a triple $ (P, \leq , f)$ such that the following are true:

    \begin{enumerate}
        \item $(P , \leq )$ is a partial order

        \item $f: P \rightarrow A$ is a function called the \emph{$A$-labeling of the heap}
    \end{enumerate}
\end{defi}

\begin{defi}
    Let $( P_{1}, \leq_{1} , f_{1})$ and $ (P_{2} , \leq_{2} , f_{2})$ be $A$-heaps. A morphism of $ A$-heaps $ \phi: ( P_{1}, \leq_{1} , f_{1}) \rightarrow (P_{2} , \leq_{2} , f_{2})$ satisfies the following properties:

    \begin{enumerate}
        \item $ \phi : P_{1} \rightarrow P_{2}$ is a function
        
        \item for all $ x,y \in P_{1}$, $ x \leq_{1} y$ implies $ \phi(x) \leq_{2} \phi(y)$

        \item for all $ x \in P_{1}$, $ f_{1}(x) = f_{2}(\phi(x))$
    \end{enumerate}
\end{defi}
Given an $ A$-heap $ (P , \leq , f)$, one has an identity morphism $ id_{P} : (P , \leq , f) \rightarrow (P , \leq , f) $ given by $ id_{P}(x) = x$ for all $ x \in P$.

\begin{defi}
    A morphism $ \phi : (P_{1}, \leq_{1} , f_{1}) \rightarrow (P_{2}, \leq_{2}, f_{2})$ of $ A$-heaps is called an \emph{isomorphism} if there exists a morphism $ \psi : (P_{2}, \leq_{2}, f_{2}) \rightarrow (P_{1}, \leq_{1}, f_{1})$ such that $ \psi \circ \phi = id_{P_{1}}$ and $ \phi \circ \psi = id_{P_{2}}$.
\end{defi}

\begin{defi}
    Let $ (W,S)$ be a Coxeter system. Let $ w\in W$. $w$ is called \emph{fully commutative} if any reduced 
expression for $w$ can be obtained from any other by means of braid relations that only involve commuting generators \cite{stembridge1996fully}.
\end{defi}
For $ n\in \N$, define $ [n] := \{ 0,1,2, \dots , n \}$.

\begin{defi}{(Stembridge \cite{stembridge1996fully})}
    Let $ (W,S)$ be a Coxeter system, and let $ w\in W$ be a fully commutative element. Let $ \textbf{\underline{w}} = (s_{1} , s_{2} , \dots , s_{n})$ be a reduced expression for $w$. Define a partial order $ \leq_{\textbf{\underline{w}}}$ on $ [n]$ given by the transitive closure of the following relation:

    $$ i \prec j \textrm{ if and only if } i < j \textrm{ and } m(s_{i}, s_{j}) \neq 2$$ 
    where $ m: S \times S \rightarrow \N \cup \{ \infty \}$ is the Coxeter matrix associated to $ (W,S)$ (and thus $ m(s_{i}, s_{j})$ is the order of $ s_{i}s_{j}$ in $ W$). Let $ f_{\textbf{\underline{w}}} : [n] \rightarrow S$ be the function given by $ f(i) = s_{i}$. Then $ ([n] , \leq_{\textbf{\underline{w}}} , f_{\textbf{\underline{w}}})$ is the \emph{$S$-heap} associated to the fully commutative element $w$.
\end{defi}

\begin{thm}{(Stembridge \cite{stembridge1996fully})} \label{welldefthm}
    Let $ (W,S)$ be a Coxeter system. Let $ w \in W$ be a fully commutative element. Let $ \textbf{\underline{w}}_{1} = (s_{1}, s_{2} , \dots , s_{n})$ and $ \textbf{\underline{w}}_{2} = (r_{1}, r_{2}, \dots , r_{n})$ be two reduced expressions for $ w$. Then the $ S$-heap associated to $w$ is well-defined up to isomorphism, meaning that $ ([n] , \leq_{\textbf{\underline{w}}_{1}}, f_{\textbf{\underline{w}}_{1}}) \cong ([n], \leq_{\textbf{\underline{w}}_{2}}, f_{\textbf{\underline{w}}_{2}})$ as $S$-heaps. 
\end{thm}

\begin{proof}
    See Proposition 2.2 of \cite{stembridge1996fully} and the discussion thereafter.
\end{proof}
If $w\in W$ is a fully commutative element, we let $ \H{w}$ denote the $ S$-heap associated to $ w$ modulo isomorphism. Note that by the previous theorem, $\H{w}$ is well-defined regardless of which reduced expression one chooses for $ w$. We let $ \leq_{w}$ and $f_{w} $ denote the partial order and $S$-labeling of $ \H{w}$ respectively

Given a Coxeter system $ (W,S)$, we let $ \weakr $ and $ \weakl$ denote the weak right order and weak left order of $ (W,S)$ respectively.

\begin{thm}{(Stembridge \cite{stembridge1996fully})} \label{fullcommweakorderthm}
    Let $ (W,S)$ be a Coxeter system. Let $ w \in W$ be a fully commutative element. If $ x\in W$ such that $ x \weakl w $ or $ x\weakr w$, then $x$ must be fully commutative.
\end{thm}

\begin{proof}
    See Proposition 2.4 of \cite{stembridge1996fully}.    
\end{proof}

\begin{prop} \label{removemaxprop}
    Let $ (W,S)$ be a Coxeter system. Let $ w$ be a fully commutative element. Consider $\H{w}$. Then the following are true:

    \begin{enumerate}
        \item Let $ D_{R}(w) : = \{ s\in S \mid \ell(ws) < \ell(w)  \}$.
        Then $ s\in D_{R}(w)$ if and only if there is a maximal vertex $ q\in \H{w}$ such that the $S$-label of $ q$ is $s$.

        \item Let $ s\in D_{R}(w)$. Then $ ws$ is fully commutative, and $ \H{ws}\cong \H{w} \setminus \{ q \}$ for some maximal vertex $q\in \H{w}$ such that the $S$-label of $q$ is $s$.

        \item Let $ q_{1}, q_{2} \in \H{w}$ denote two distinct maximal vertices. Then the $S$-labels of $ q_{1}$ and $q_{2}$ must be distinct.

        \item Let $ q \in \H{w}$ be a maximal element with $ S$-label $s$. Then $ ws$ is fully commutative and $ \H{ws} \cong \H{w} \setminus \{ q  \}$.

        \item Let $ x,y\in W$ such that $ x$ and $ y$ are both fully commutative elements. Then $ x=y$ if and only if $ \H{x} \cong \H{y}$.
    \end{enumerate}
\end{prop}

\begin{proof}
    For part (1), it is a well-known fact that $ s\in D_{R}(w)$ if and only if there is a reduced expression $ \textbf{\underline{w}} = (s_{1}, s_{2} , \dots , s_{n})$ for $w$ where $ s_{n} =s$. 
    
    Suppose that $ \textbf{\underline{w}} = (s_{1}, s_{2} , \dots , s_{n})$ is a reduced expression for $ w$ with $ s_{n} = s$. Then when one constructs the $S$-heap for $w$ by using the reduced expression $ \textbf{\underline{w}}$, the element $ n\in [n]$ of the heap $ ([n], \leq_{\textbf{\underline{w}}}, f_{\textbf{\underline{w}}})$ will have an $ S$-label of $ s$. By the definition of $ \leq_{\textbf{\underline{w}}}$, $ n \in [n]$ will be a maximal element. Thus, we have shown the existence of a maximal $ q\in \H{w}$ such that the $S$-label of $ q$ is $s$.  

    Suppose now that there exists some maximal $ q\in \H{w}$ such that the $ S$-label of $ q$ is $ s$. Consider any reduced expression $ \textbf{\underline{w}} = (s_{1}, s_{2}, \dots , s_{n})$ for $ w$. Note that since $ q$ has an $S$-label of $ s$, there must exist at least one positive integer $i$ such that $ s_{i} = s$. If $ i=n$, then we have constructed a reduced expression $\textbf{\underline{w}}$ with $ s_{n} = s$. Thus, suppose that $ i \neq n$. Let $ j$ be the largest positive integer such that $ s_{j} = s$. Consider $ s_{j+1}, s_{j+2}, \dots , s_{n}$. If there existed a positive integer $ k >j$ with $ m(s_{j} , s_{k}) \neq 2$, then because $ j$ is the largest positive integer such that $ s_{j} = s$, the definition of $ \leq_{\textbf{\underline{w}}}$ would imply that every vertex $ p \in \H{w}$ whose $S$-label is $ s$ would \emph{not} be a maximal element of $ \H{w}$. This contradicts our assumption that there exists at least one maximal $ q\in \H{w}$ whose $S$-label is $s$. Hence, $ m(s_{j}, s_{k}) = 2$ for all $ k > j$. Hence, $ (s_{j}s_{k})^{2} = 1$ for all $ k> j$, which implies that $ s_{j}s_{k} = s_{k}s_{j}$ for all $k> j$. Therefore, we can move $ s = s_{j}$ to the right end of a reduced expression for $ w$ by commuting $ s_{j}$ with $ s_{k}$ for $ k > j$. This finishes the proof of claim (1).

    Let $ s \in D_{R}(w)$. Thus, there exists some reduced expression $ \textbf{\underline{w}} = (s_{1}, s_{2} , \dots , s_{n})$ with $ s_{n}= s$. Consider $ ([n], \leq_{\textbf{\underline{w}}}, f_{\textbf{\underline{w}}})$. Note that $ n\in [n]$ is maximal in $([n], \leq_{\textbf{\underline{w}}}, f_{\textbf{\underline{w}}}) $ with $ S$-label $s$. Note also that $ \textbf{\underline{ws}} : = (s_{1}, s_{2} , \dots , s_{n-1})$ is a reduced expression for $ ws$, and note further that $ ws$ must be fully commutative by Theorem \ref{fullcommweakorderthm} since $ ws \sweakr w$. Hence, the $ S$-heap $([n-1], \leq_{\textbf{\underline{ws}}}, f_{\textbf{\underline{ws}}})$ is obtained from the $ S$-heap $([n], \leq_{\textbf{\underline{w}}}, f_{\textbf{\underline{w}}}) $ by removing a maximal element with $ S$-label $s$. Thus, $ \H{ws}\cong \H{w} \setminus \{ q \}$ for some maximal $ q\in \H{w}$ with $ S$-label $s$. This proves claim (2).

    Let $ q_{1}, q_{2} \in \H{w}$ be two distinct maximal vertices. Let $ r$ and $ s$ denote the $S$-labels of $ q_{1}$ and $ q_{2}$ respectively. By part (1), we know that $ r \in D_{R}(w)$. By part (2), we know that there must exist some maximal vertex $ q_{3} \in \H{w}$ with $ S$-label $r$ such that $ \H{wr} \cong \H{w} \setminus \{ q_{3} \}$. Note also that $ \ell(wr) < \ell(w)$ where $ r \notin D_{R}(wr)$. If $ q_{1} \neq q_{3}$, then $ q_{1} \in \H{w} \setminus \{ q_{3} \} \cong \H{wr}$ is maximal with $ S$-label $r$. Hence, $r \in D_{R}(wr)$, which contradicts $ r \notin D_{R}(wr) $, and therefore $ q_{1} = q_{3}$. Thus, $ \H{wr} \cong \H{w} \setminus \{ q_{1} \}$. But now observe that $ q_{2} \in \H{w} \setminus \{ q_{1} \} \cong \H{wr}$ is maximal with $ S$-label $s$. Hence, $ \ell(wrs) < \ell(wr)$. Since $ \ell(wr) < \ell(w)$, we deduce that $ \ell(wrs) < \ell(w)$. If $ r=s$, then $ wrs = w$, and thus we would deduce $ \ell(w) < \ell(w)$, which is a contradiction. Thus, $ r\neq s$, and this proves claim (3).

    Let $ q \in \H{w}$ be a maximal element with $S$-label $s$. By part (1), we know that $ s \in D_{R}(w)$. Hence, $ ws \sweakr w$, so it follows that $ ws$ is fully commutative. By part (2), we know that there must exist some maximal $ q_{1} \in \H{w}$ with $ S$-label $s$ such that $ \H{ws} \cong \H{w} \setminus \{ q_{1} \} $. By part (3), we know that $ q = q_{1}$. Hence, $ \H{ws} \cong \H{w} \setminus \{ q \}$, which proves part (4).

    Let $ x,y\in W$ be such that $ x$ and $ y $ are both fully commutative elements. Since Theorem \ref{welldefthm} proved that $ \H{x}$ and $ \H{y}$ are well-defined, it follows that $ x=y$ implies $ \H{x} \cong \H{y}$. Suppose now that $ \H{x} \cong \H{y}$. If $ \H{x} \cong \H{y} = \emptyset$, then $ x=y=1$. Thus, let us suppose that $ \H{x} \cong \H{y} \neq \emptyset$. Let $ q \in \H{x} \cong \H{y}$ be a maximal element, and let $ s$ denote the $S$-label of $ q$. By part (4), we deduce that $ \H{xs} \cong \H{x} \setminus \{  q \}$ and $ \H{ys} \cong \H{y} \setminus \{ q \}$. Thus, we deduce that $\H{xs} \cong \H{x} \setminus \{  q \} \cong \H{y} \setminus \{ q \} \cong \H{ys} $. But since $ \H{xs} \cong \H{ys}$ where $ |\H{xs}| < |\H{x}|$, we can apply induction to conclude that $ xs=ys$, which implies $ x=y$. This proves part (5).
\end{proof}

\begin{prop} \label{removeminprop}
    Let $ (W,S)$ be a Coxeter system. Let $ w$ be a fully commutative element. Consider $\H{w}$. Then the following are true:

    \begin{enumerate}
        \item Let $ D_{L}(w) : = \{ s\in S \mid \ell(sw) < \ell(w)  \}$.
        Then $ s\in D_{L}(w)$ if and only if there is a minimal vertex $ p\in \H{w}$ such that the $S$-label of $ p$ is $s$.

        \item Let $ s\in D_{L}(w)$. Then $ sw$ is fully commutative, and $ \H{sw}\cong \H{w} \setminus \{ p \}$ for some minimal vertex $p\in \H{w}$ such that the $S$-label of $p$ is $s$.

        \item Let $ p_{1}, p_{2} \in \H{w}$ denote two distinct minimal vertices. Then the $S$-labels of $ p_{1}$ and $p_{2}$ must be distinct.

        \item Let $ p \in \H{w}$ be a minimal element with $ S$-label $s$. Then $ sw$ is fully commutative and $ \H{sw} \cong \H{w} \setminus \{ p  \}$.

    \end{enumerate}
\end{prop}

\begin{proof}
    Use the general fact that $ s\in D_{L}(w)$ if and only if there is a reduced expression $ \textbf{\underline{w}} = (s_{1}, s_{2}, \dots , s_{n})$ for $ w$ where $ s_{1} = s$. Also note that if $ s\in D_{L}(w)$, then $sw \sweakl w$, and therefore $ sw$ must be fully commutative by Theorem \ref{fullcommweakorderthm}. Now mirror the proof of Proposition \ref{removemaxprop}.
\end{proof}
Let $ (W,S)$ be a Coxeter system with $ w\in W$ being fully commutative. A subset $ I \subseteq \H{w}$ is called a \emph{labeled order ideal} if for any $ q\in I$, $ p \leq_{w} q$ implies that $ p \in I$. The $S$-labeling of $ \H{w}$ naturally restricts to an $S$-labeling of $I$. Hence, $I$ can be viewed as an $S$-heap modulo isomorphism. Note that if $I, J \subseteq \H{w}$ are labeled order ideals, then $I \cup J \subseteq \H{w}$ is also a labeled order ideal.

\begin{prop} \label{existscorrprop}
    Let $ (W,S )$ be a Coxeter system, and let $ w\in W$ be a fully commutative element. Let $I \subseteq \H{w}$ be a labeled order ideal of $ \H{w}$. Then there exists a unique $ x\in W$ such that the following are true:

    \begin{enumerate}
        \item $ x \weakr w$, and $ x$ is fully commutative

        \item $ \H{x} \cong I$
    \end{enumerate}
\end{prop}

\begin{proof}
    If such an $ x$ exists, note that part (5) of Proposition \ref{removemaxprop} will immediately prove uniqueness. We now prove the existence of such an $ x$ via induction on $ |\H{w} \setminus I| = n$. Note that if $|\H{w} \setminus I| = 0 $, then $ I = \H{w}$ and clearly $ x =w$ satisfies the desired properties.

    Suppose now that $ |\H{w} \setminus I| = n+1$ for some  $n \in \N$. Hence, $ I \subsetneq \H{w}$. Since $I$ is an order ideal of $ \H{w}$ with $I \subsetneq \H{w}$, it follows that there must exist some maximal element $ q\in \H{w}$ such that $ q \notin I$. Let $ s$ denote the $S$-label of $q$. By part (4) of Proposition \ref{removemaxprop}, one has that $ ws \sweakr w$, $ws $ is fully commutative, and $ \H{ws} \cong \H{w} \setminus \{ q \}$. Since $I$ is a labeled order ideal of $ \H{w}$ with $ q\notin I$, $I$ is naturally viewed as a labeled order ideal of $ \H{ws}$. Hence, $I$ is a labeled order ideal of $ \H{ws}$, and $ |\H{ws} \setminus I| =n$. The induction hypothesis allows us to construct a fully commutative $ x \in W$ such that $ x \weakr ws \sweakr w$ and $ \H{x} \cong I$. This completes the induction and finishes the proof.

\end{proof}

\begin{prop} \label{mustbecompprop}
    Let $ (W,S)$ be a Coxeter system, and let $ w\in W$ be fully commutative. Let $ q_{1}, q_{2} \in \H{w}$ be two distinct elements with the same $S$-label. Then $ q_{1} \leq_{w} q_{2}$ or $ q_{2} \leq_{w} q_{1}$.
\end{prop}

\begin{proof}
    Suppose for the sake of contradiction that $ q_{1}$ and $ q_{2}$ were incomparable with respect to $ \leq_{w}$. For $j =1,2$, define

    $$ I(q_{j}) : = \{  p\in \H{w} \mid p \leq_{w} q_{j} \}$$
    Note that $I(q_{1})$ and $ I(q_{2})$ are labeled order ideals of $ \H{w}$. Since the union of labeled order ideals is a labeled order ideal, it follows that $ J : = I(q_{1}) \cup I(q_{2})$ is a labeled order ideal of $ \H{w}$. By Proposition \ref{existscorrprop}, there must exist a unique $ x \weakr w$ such that $\H{x} \cong J$. But since $ q_{1}$ and $ q_{2}$ are incomparable with respect to $\leq_{w}$, it follows that $q_{1}$ and $q_{2}$ are distinct maximal elements of $ J \cong \H{x}$ with the same $S$-label. This contradicts part (3) of Proposition \ref{removemaxprop}.
\end{proof}

\begin{prop} \label{equividealprop}
    Let $ (W,S)$ be a Coxeter system. Let $ w\in W$ be a fully commutative element. Let $x\in W $. Then the following two conditions are equivalent:

    \begin{enumerate}
        \item $ x \weakr w$

        \item $ x$ is fully commutative, and there exists a labeled order ideal $ I \subseteq \H{w}$ such that $ \H{x} \cong I$ (as $S$-heaps modulo isomorphism).
    \end{enumerate}
\end{prop}

\begin{proof}
    Suppose that $ x \weakr w$. Let $ \textbf{\underline{x}} = (s_{1}, s_{2}, \dots , s_{k})$ denote a reduced expression for $ x$. Since $ x\weakr w$, it follows that we can extend $\textbf{\underline{x}}$ to a reduced expression for $w$, or in other words, there must exist an $n \geq k$ such that $ \textbf{\underline{w}} = (s_{1}, s_{2}, \dots , s_{k}, \dots, s_{n})$. For $ j \in \N$ such that $ k \leq j \leq n$, define $ w_{j} : = s_{1}s_{2} \dots s_{j}$. Thus, we get a sequence of elements $ w_{k}, w_{k+1}, \dots , w_{n-1}, w_{n}$ where $w_{k} : = x$ and $ w_{n}: = w$. Because $ w_{j} \weakr w$ for all $ j$ such that $ k \leq j \leq n$, it follows from Theorem \ref{fullcommweakorderthm} that $ w_{j}$ is fully commutative for all $ j$ such that $ k \leq j \leq n$ (and in particular, $ w_{k} = x$ is fully commutative). Furthermore, note that for $ k < j \leq n$, we have $ w_{j-1} = w_{j}s_{j}$ where $ s_{j} \in D_{R}(w_{j})$. Hence, by Proposition \ref{removemaxprop} part (2), it follows that $ \H{w_{j-1}} \cong \H{w_{j}} \setminus \{  q_{j} \}$ where $ q_{j}$ is some maximal vertex of $ \H{w_{j}}$. Thus, one deduces that $ \H{x}$ is obtained from $ \H{w}$ by successively removing maximal elements. Therefore, $ \H{x}$ must be a labeled order ideal of $ \H{w}$.

    Suppose now that $ x$ is fully commutative and that there exists a labeled order ideal $ I \subseteq \H{w}$ such that $ \H{x} \cong I$. By Proposition \ref{existscorrprop}, one can construct a fully commutative $y \weakr w$ such that $ \H{y} \cong I$. But because $ \H{y} \cong I \cong \H{x}$, Proposition \ref{removemaxprop} part (5) implies that $ x = y\weakr w$.
\end{proof}

\begin{defi}
    Let $ (W,S)$ be a Coxeter system, and let $w\in W$ be fully commutative. Let $ p,q \in \H{w}$. We define a relation $ p \vartriangleleft_{w} q$ to mean that $ p <_{w} q$ and that there does not exist any $ z \in \H{w}$ such that $ p <_{w} z <_{w} q $. If $ p \vartriangleleft_{w} q$, we say that $ q$ \emph{covers} $ p$.
\end{defi}

\begin{defi}
    Let $ (W,S)$ be a Coxeter system, and let $w\in W$ be fully commutative. Let $ q \in \H{w}$. \emph{A chain defined at $ q$} is a sequence of coverings

    $$ p_{1} \vartriangleleft_{w} p_{2} \vartriangleleft_{w} p_{3} \vartriangleleft_{w} \dots \vartriangleleft_{w} p_{n-1} \vartriangleleft_{w}p_{n}  = q$$
    The \emph{length} of such a chain is defined to be $n$. The length of the empty chain is defined to be $0$.
\end{defi}

\begin{defi}
    Let $ (W,S)$ be a Coxeter system, and let $w\in W$ be fully commutative. Let $ q \in \H{w}$. \emph{A chain of maximal length defined at $q$} is a chain defined at $q$ whose length is maximal (among the chains defined at $q$).
\end{defi}
Let $ w\in W$ be fully commutative. Let $ q \in \H{w}$. Since $ \H{w}$ is a finite poset, it follows that a chain of maximal length defined at $q$ exists, but such a chain need not be unique.

\begin{prop} \label{uniqueidealprop}
    Let $ (W,S)$ be a Coxeter system. Let $ w\in W$ be a fully commutative element. Let $x\in W $. If $ x \weakr w$, then $x$ is fully commutative and there exists a \emph{unique} order ideal $I \subseteq \H{w}$ such that $ \H{x} \cong I$.
\end{prop}

\begin{proof}
    Since $ x \weakr w$, Proposition \ref{existscorrprop} implies that $x$ is fully commutative and that $ \H{x} \cong I$ for some labeled order ideal $I \subseteq \H{w}$. We just need to prove the uniqueness of $I \subseteq \H{w}$.

    Suppose for the sake of contradiction that $I, J \subseteq \H{w}$ are two distinct labeled order ideals of $ \H{w}$ such that $ \H{x} \cong I$ and $ \H{x} \cong J$. Since $I$ and $J$ are distinct labeled order ideals of $ \H{w}$, it follows that at least one of the following must be true: $ I \nsubseteq J$ or $ J \nsubseteq I$. Without loss of generality, let us assume that $ I \nsubseteq J$. Again, since we are dealing with labled order ideals, the condition $ I \nsubseteq J$ implies that there must exist some $ q_{I}\in I \setminus J$ which is maximal in $I$. Since $I \cong \H{x} \cong J$, let $ \phi: I \rightarrow J$ denote an isomorphism of $S$-heaps. Define $ q_{J} : = \phi(q_{I}) \in J$. Since $ q_{I}, q_{J} \in \H{w}$ have the same $S$-label, it follows from Proposition \ref{mustbecompprop} that either $ q_{I} \leq_{w} q_{J}$ or $ q_{J} \leq_{w} q_{I}$. But note that if $ q_{I} \leq_{w} q_{J}$, then because $ J$ is an order ideal of $ \H{w}$, it would follow that $ q_{I} \in J$, which contradicts $ q_{I} \in I \setminus J$. Hence, we deduce that $ q_{J} <_{w} q_{I}$ in $\H{w}$. Note that since $ I$ is a labeled order ideal of $ \H{w}$, it follows that any chain defined at $ q_{I}$ in $ \H{w}$ can be viewed as a chain defined at $ q_{I}$ in $I$, and vice versa. Similarly, since $J$ is a labeled order ideal of $ \H{w}$, any chain defined at $ q_{J}$ in $\H{w}$ can be viewed as a chain defined at $ q_{J}$ in $J$, and vice versa. Since $ q_{J} : = \phi(q_{I})$ where $ \phi: I \rightarrow J$ is an isomorphism of $ S$-heaps, it follows that there is a bijective correspondence between chains of maximal length defined at $ q_{I}$ in $I$ and chains of maximal length defined at $ q_{J}$ in $J$ (since isomorphisms of $S$-heaps preserve the covering relation). More specifically, the length of a chain of maximal length defined at $ q_{I}$ in $I$ should have the same length as a chain of maximal length defined at $ q_{J}$ in $J$. Let $n$ denote the length of a chain of maximal length defined at $ q_{J}$ in $J$ (or equivalently, the length of a chain of maximal length defined at $q_{I}$ in $I$). But recall that since $ I$ and $J$ are order ideals of $ \H{w}$, chains of maximal length defined at $q_{I}$ in $ I$ and $ q_{J}$ in $J$ can be viewed as chains of maximal length defined at $ q_{I}$ in $\H{w}$ and $ q_{J}$ in $ \H{w}$ respectively. Hence, we deduce that any chain of maximal length defined at $ q_{I}$ in $ \H{w}$ should have the same length $ n$ as a chain of maximal length defined at $ q_{J}$ in $\H{w}$. But recall that $ q_{J} <_{w} q_{I}$ in $ \H{w}$. Thus, any chain defined at $ q_{J}$ can be extended to an even longer chain defined at $ q_{I}$. In other words, if we take a chain of maximal length defined at $ q_{J}$ in $\H{w}$ (whose length is $n$), we would be able to extend it to a chain defined at $ q_{I}$ in $ \H{w}$ whose length is strictly greater than $n$. But chains of maximal length defined at $ q_{I}$ in $\H{w}$ should have length exactly $n$. This is a contradiction, so we must conclude that $ I = J$.
\end{proof}

\begin{prop} \label{joinidealprop}
    Let $ (W,S)$ be a Coxeter system. Let $ w\in W$ be fully commutative. Suppose that there exist fully commutative elements $ x,y \in W$ and labeled order ideals $ I , J \subseteq \H{w}$ such that

    \begin{itemize}
        \item $ \H{x} \cong I$ and $ \H{y} \cong J$

        \item $ \H{w} = I \cup J$
    \end{itemize}
    Then $ w = x \vee y$ where $ \vee $ denotes the join in the weak right order of $(W,S)$.
\end{prop}

\begin{proof}
    By Proposition \ref{equividealprop}, we conclude that $ x\weakr w$ and $ y \weakr w$. Since the weak right order is a complete meet-semilattice, the conditions $ x \weakr w$ and $ y \weakr w$ imply that $ x \vee y$ exists and $ x \vee y \weakr w$ (and hence $ x\vee y$ is fully commutative by Theorem \ref{fullcommweakorderthm}). 
    
    If $ x \vee y = w$, then we are done. So suppose for the sake of contradiction that $ x\vee y \neq w$, and hence $ x \vee y \sweakr w$. Because $ x \vee y \sweakr w$, it follows that there exists an $ s\in D_{R}(w)$ such that $ x\vee y \weakr ws$. By Proposition \ref{removemaxprop} part (4), we have that $ \H{ws} \cong \H{w} \setminus \{ q \}$ where $q\in \H{w}$ is the unique maximal element of $ \H{w}$ such that the $S$-label of $ q$ is $s$. Thus, we can naturally view $ \H{ws}$ as a proper labeled order ideal of $ \H{w}$ where $ q \notin \H{ws}$. Note that since $ \H{w} = I \cup J$, it follows that $ q \in I$ or $ q\in J$. Without loss of generality, let us assume that $ q\in I$. But because $ x \weakr x\vee y \weakr ws$, it follows from Proposition \ref{equividealprop} that there exists some labeled order ideal $K \subseteq \H{ws} \subseteq \H{w}$ such that $ \H{x} \cong K$. But note that $ K \neq I$ since $ q\in I$ but $ q\notin K$ since $ K \subseteq \H{ws}$ and $ q\notin \H{ws}$. $I$ and $K$ are two distinct labeled order ideals of $ \H{w}$ such that $ \H{x} \cong I$ and $ \H{x} \cong K$. This contradicts Proposition \ref{uniqueidealprop}. Hence, $ w = x\vee y$.
\end{proof}

\begin{prop} \label{joinisinvolution}
    Let $ (W,S)$ be a Coxeter system, and let $ w\in W$ be fully commutative. Let $ x,y\in W$ be such that $ x \weakr w$ and $ y\weakr w$. Suppose further that $ x^{2} = 1$ and $ y^{2}=1$. Then $ x\vee y$ exists and $ (x \vee y)^{2}  =1$.
\end{prop}

\begin{proof}
    Since $ x \weakr w$ and $ y \weakr w$, there exists labeled order ideals $I, J \subseteq \H{w}$ such that $\H{x} \cong I $ and $ \H{y} \cong J$ by Proposition \ref{equividealprop}. The union of two labeled order ideals is also a labeled order ideal, so $I \cup J \subseteq \H{w}$ is a labeled order ideal. By Proposition \ref{existscorrprop}, there exists some $ z \weakr w$ such that $ \H{z} \cong I \cup J$. By Proposition \ref{joinidealprop}, we must have that $ z = x \vee y$. Since $ \H{x\vee y} \cong I \cup J$, let us naturally view $ I$ and $J$ as labeled order ideals of $ \H{x\vee y}$ and thus $ \H{x \vee y} = I \cup J$.

    We now prove the claim that $ (x\vee y)^{2} = 1$ via induction on $n= \ell(x) + \ell(y)$. Note that if $ n = 0$, then $ x=1$ and $ y=1$, and hence $ x\vee y = 1$, and thus the desired conclusion is true.

    Suppose now that the claim holds for all $ k =0,1,2,  \dots , n$ where $ k = \ell(x) + \ell(y)$. We wish to prove that the claim holds for $ n+1 = \ell(x) + \ell(y)$. Note that because $ \ell(x) + \ell(y) = n+1$ where $n\in \N$, it follows that at least one of $ x$ or $y$ is not the identity, and hence $ x\vee y$ is not the identity. Thus, $ \H{x \vee y} \neq \emptyset$. 
    
    Let us suppose that all vertices of $ \H{x \vee y}$ are minimal. By Proposition \ref{removeminprop} part (3), this would imply that each vertex has a distinct $S$-labeling. Furthermore, by definition of $ \leq_{x\vee y}$, if $ r$ and $s$ were the $S$-labels of distinct vertices in $ \H{x\vee y}$, then because all vertices are minimal, we would necessarily have that $ m(r,s)  =2$ (if $m(r,s)\neq 2$, then this would imply that one vertex is strictly greater than the other with respect to $ \leq_{x\vee y}$, contradicting minimality of all vertices). But because $ m(r,s) = 2$, this implies that $ (rs)^{2} = 1$ and hence $ rs=sr$. Thus, if all vertices of $ \H{x\vee y}$ are minimal, then $ x\vee y$ would be a product of distinct simple reflections that pairwise commute. Therefore, we would have that $ (x\vee y)^{2} = 1$. Thus, from now on, we can assume that not all vertices of $ \H{x\vee y}$ are minimal. In particular, since we are dealing with finite posets, there must exist some maximal vertex $ q \in \H{x\vee y}$ such that $ q$ is not minimal in $ \H{x\vee y}$. Since $ \H{x\vee y} = I \cup J$, we must have that $ q\in I$ or $ q\in J$. Without loss of generality, let us assume that $ q\in I$. From now on, let $ s\in S$ denote the $S$-labeling of $q$. Since $ q$ is maximal in $ \H{x\vee y} = I \cup J$ where $ q\in I$, it follows that $ q$ is also maximal in $I$. But if $ q$ is maximal in $I$ and $ \H{x} \cong I$, then it follows that $ s\in D_{R}(x)$ by Proposition \ref{removemaxprop} part (1). But since $ x^{2}=1$, we have that $ D_{R}(x) = D_{L}(x)$, so we conclude that there must exist some minimal vertex $ p \in I$ such that the $S$-label of $p$ is $s$. Because $ I$ is an order ideal of $\H{x\vee y}$, it follows that $p$ is also minimal in $ \H{x\vee y}$. Note that $ p \neq q$ since $ q$ is assumed to be maximal but not minimal in $ \H{x\vee y}$. Furthermore, since $p$ and $ q$ have the same $S$-label in $ \H{x\vee y}$ it follows that $p <_{x\vee y} q$ by Proposition \ref{mustbecompprop}.

    \textbf{Case 1: $p\notin J$}

    Suppose that $ p \notin J$. Not that since $J$ is a labeled order ideal, and because $ p <_{x\vee y} q$, it follows that $ q\notin J$. Hence, $ \H{x\vee y} \setminus \{ p,q \} = (I \setminus \{  p,q \} ) \cup J$. By applying Proposition \ref{removemaxprop} part (4) and then applying Proposition \ref{removeminprop} part (4), we deduce that $ \H{s(x\vee y)s} \cong \H{x\vee y} \setminus \{ p,q \} $. Note also that $ (I \setminus \{  p,q \} )$ and $J$ are still labeled order ideals of $\H{s(x\vee y)s} \cong \H{x\vee y} \setminus \{ p,q \} $. Furthermore, by applying Proposition \ref{removemaxprop} part (4) and then applying Proposition \ref{removeminprop} part (4), we deduce that $ \H{sxs} \cong I \setminus \{ p,q \}$. We still have $ \H{y} \cong J$. Thus, since $ \H{s(x\vee y )s} = (I \setminus \{  p,q \} ) \cup J$, we can apply proposition \ref{joinidealprop} to conclude that $ s(x\vee y)s = (sxs)\vee y$. But note that $ s(x\vee y)s$ is a fully commutative element where $ sxs \weakr s(x\vee y)s$, $ y \weakr s(x\vee y)s$, $ (sxs)^{2}=1$, $ y^{2}=1$, and $ \ell(sxs) + \ell(y) = \ell(x) -2 +\ell(y) =n+1 - 2 = n-1< n+1$. Hence, we can apply induction to conclude that $ (s(x\vee y)s)^{2} = 1$, which implies $ (x\vee y)^{2} = 1$. 

    \textbf{Case 2: $ p \in J$ and $ p$ is maximal in $J$}

    Suppose now that $ p\in J$ and $ p$ is maximal in $J$. Since $p$ is maximal in $J$, $ q\notin J$. Note also that because $ p $ is minimal in $\H{x\vee y}$, it follows that $ p$ is minimal in $J$. But because $p$ is both maximal and minimal in $J$, it follows that $p$ is incomparable to all other vertices of $ J$. Let $ r$ denote the $S$-label of some arbitrary vertex in $ J \setminus \{ p \}$. Note that because $ p$ is incomparable to all vertices in $ J \setminus \{ p \}$, it follows that $ m(r,s) = 2$ (for if $ m(r,s) \neq 2$ for some $ r$ where $r$ is an $S$-label of some vertex in $ J \setminus \{ p \}$, then $ p$ would be comparable to some element in $ J \setminus \{ p \}$ by definition of $ \leq_{x\vee y}$). Since $ \H{y} \cong J$, this implies that in any given reduced expression for $y$, the simple reflection $ s$ will commute with any other simple reflection appearing in that reduced expression. In particular, $ sy=ys$. Note that since $ p\in J$ and $ q\notin J$, we have that $ \H{x\vee y} \setminus \{  p,q \} = (I \setminus \{ p,q \}) \cup (J \setminus \{ p\})$. Note that $I \setminus \{  p,q\} $ and $ J \setminus \{ p\}$ are labeled order ideals of $ \H{x\vee y} \setminus \{  p,q \}$. By applying Proposition \ref{removemaxprop} part (4) and then applying Proposition \ref{removeminprop} part (4), we deduce that $ \H{s(x\vee y)s} \cong \H{x\vee y} \setminus \{ p,q \} $. Similarly, $ \H{sxs} \cong I \setminus \{ p,q \}$. By applying Proposition \ref{removemaxprop} part (4), we also deduce that $ \H{ys} \cong J \setminus \{ p \}$. Hence, we can apply Proposition \ref{joinidealprop} to conclude that $ s(x\vee y)s = (sxs) \vee (ys)$. But note that $ s(x\vee y)s$ is a fully commutative element where $sxs \weakr s(x\vee y)s $, $ ys \weakr s(x\vee y)s$, $ (sxs)^{2}  =1$, $ (ys)^{2} = 1$, $ \ell(sxs)+ \ell(ys) = \ell(x) - 2 + \ell(y) -1 = n+1 -3 = n-2< n+1$. Hence, we can apply induction to conclude that $ (s(x\vee y)s)^{2} = 1$, which implies that $ (x\vee y)^{2} = 1$.

    \textbf{Case 3: $p \in J$ and $ p$ is not maximal in $J$}

    Suppose now that $ p\in J$ and $ p$ is not maximal in $ J$. Because $p$ is minimal in $ \H{x\vee y}$, $ p$ is also minimal in $J$ since $J$ is a labeled order ideal of $ \H{x\vee y}$. Since $ p$ is minimal in $J$ with $ \H{y} \cong J$, it follows from Proposition \ref{removeminprop} part (1) that $s \in D_{L}(y) $. But note that because $y^{2}  =1$, it follows that $ D_{R}(y) = D_{L}(y)$. Thus, $ s \in D_{R}(y)$. But by Proposition \ref{removemaxprop} part (1), this implies that there exists some vertex $ z \in J \cong \H{y}$ such that $ z$ is maximal in $J$ and $ z$ has $S$-label $s$. Furthermore, $ z\neq p$ since $ p\in J$ is minimal but not maximal in $J$. Note that since $ z$ and $ q$ have the same $ S$-label in $ \H{x\vee y}$, it follows from Proposition \ref{mustbecompprop} that $ z \leq_{x\vee y} q$ or $ q \leq_{x\vee y} z$. But since $ q\in \H{x\vee y}$ is maximal, we conclude that $ z \leq_{x\vee y} q$. Since $ q\in I$ and $I$ is a labeled order ideal of $ \H{x\vee y}$, we conclude that $ z\in I$. Note that if $ z = q$, then one has $ \H{x\vee y} \setminus \{ p,q \} = (I \setminus \{  p,q \}) \cup (J \setminus \{  p,z \} ) $. If $ z\neq q$, then $ z <_{x\vee y} q$, and because $ z$ is maximal in $J$, we have that $ q\notin J$. Furthermore, because $ q\in I$, we still have $ \H{x\vee y} \setminus \{ p,q \} = (I \setminus \{ p,q \} ) \cup (J \setminus \{ p,z \}) $. Thus, we have shown that regardless of whether $ z=q$ or $ z\neq q$, we have $ \H{x\vee y} \setminus \{ p,q \} = (I \setminus \{ p,q \} ) \cup (J \setminus \{ p,z \}) $. Note that $I \setminus \{ p,q \} $ and $ J \setminus \{ p,z \}$ are labeled order ideals of $ \H{x\vee y} \setminus \{ p,q \}$. One can use Proposition \ref{removemaxprop} part (4) and Proposition \ref{removeminprop} part (4) to conclude that $ \H{s(x\vee y)s} \cong \H{x\vee y}\setminus \{ p,q \}$, $ \H{sxs} \cong I \setminus \{ p,q \}$, and $ \H{sys} \cong J \setminus \{  p,q \}$. Hence, one can Proposition \ref{joinidealprop} to deduce that $ s(x\vee y)s = (sxs)\vee (sys)$. But note that $ s(x\vee y)s$ is a fully commutative element where $ sxs \weakr s(x\vee y)s$, $ sys\weakr s(x\vee y)s$, $ (sxs)^{2} = 1$, $ (sys)^{2}=1$, and $ \ell(sxs)+\ell(sys) = \ell(x) - 2 + \ell(y) -2 = n+1 - 4 = n-3 < n+1$. Hence, we can apply induction to conclude that $ (s(x\vee y)s)^{2} = 1$, and thus $ (x\vee y)^{2} = 1$.

    The proof of the three cases above along with the proof of the case when $ \H{x\vee y}$ is composed solely of minimal elements exhausts all possible cases. Hence, we deduce that $ (x\vee y)^{2} = 1$ in general, which concludes that proof of the proposition.

\end{proof}

\begin{proof}[Proof of Theorem \ref{solvedconj}]
    Let $ (W,S)$ be a Coxeter system, and let $ w\in W$ be fully commutative such that $ w\neq 1$. Suppose for the sake of contradiction that there exist two distinct non-trivial involutions of maximal length $ x, y\in [1,w]$. Because $ x,y\in [1,w]$, it follows that $ x\weakr w$ and $ y\weakr w$. Hence, we can apply Proposition \ref{joinisinvolution} to conclude that $ x\vee y$ exists, $ x\vee y \weakr w$, and that $ x\vee y$ is an involution. But since $ x\neq y$ and $ \ell(x) = \ell(y)$, we deduce that $ \ell(x\vee y) > \ell(x) = \ell(y)$. Hence, $ x\vee y$ is an involution in $ [1,w]$, and $\ell(x\vee y)> \ell(x) = \ell(y)$. This contradicts the property that $ x$ and $ y$ are involutions of maximal length in $ [1,w]$. Therefore, any non-trivial involution of maximal length in $ [1,w]$ must be unique, and hence $ w$ satisfies the ancestor property.
\end{proof}

\begin{defi}
Let $  (W,S)$ be a Coxeter system, and let $ w\in W$ be any element. Define $ N(w) : = \{ t\in T \mid \ell(tw) < \ell(w)  \}$.
\end{defi}

\begin{prop} \label{weakrequivreflcocycle}
    Let $ (W,S)$ be a Coxeter system, and let $ x,y\in W$ be any two elements. Then:

    $$ x\weakr y \textrm{ if and only if } N(x) \subseteq N(y)$$
\end{prop}

\begin{proof}
    See Proposition 3.1.3 of \cite{bjorner2005combinatorics}.
\end{proof}

\begin{defi}
    Let $ (W,S)$ be a Coxeter system, and let $ w\in W$ be fully commutative. A \emph{principle labeled order ideal} $I$ of $ \H{w}$ is a labeled order ideal such that there exists a vertex $ q \in \H{w}$ where

    $$ I : = \{ p \in \H{w} \mid  p \leq_{w} q  \}$$
    If $ q\in \H{w}$, we let $ I(q) \subseteq \H{w}$ denote the principle labeled order ideal defined by $ q\in \H{w}$.
\end{defi}

\begin{prop} \label{welldefreflection}
    Let $ (W,S )$ be a Coxeter system, and let $w\in W$ be fully commutative. Let $ q\in \H{w}$. Consider the principle labeled order ideal $I(q) \subseteq \H{w}$. Let $ x\in W$ be the unique element such that $ x \weakr w$ and $ \H{x} \cong I(q)$ (Proposition \ref{existscorrprop}). Then there is a well-defined reflection $ t\in T$ such that:

    $$ t = s_{1}s_{2} \dots s_{k-1}s_{k}s_{k-1}\dots s_{2}s_{1}$$
    where $ \textbf{\underline{x}} = (s_{1},s_{2}, \dots, s_{k})$ is some reduced expression for $ x$.
\end{prop}

\begin{proof}
    Consider two reduced expressions for $ x$: $ (s_{1},s_{2}, \dots, s_{k})$ and $ (r_{1}, r_{2}, \dots , r_{k})$. Hence, we have that $ s_{1}s_{2}\dots s_{k} = x = r_{1}r_{2} \dots r_{k}$. But note that since $ q \in \H{x} \cong I(q)$ is the unique maximal element of $ \H{x}$, it follows that $ D_{R}(x)$ is a singleton set (Proposition \ref{removemaxprop} part (1)), and hence $ s_{k} = r_{k}$. But if $ s_{k} = r_{k}$ and $ s_{1}s_{2}\dots s_{k} = r_{1}r_{2} \dots r_{k}$, we conclude that $ s_{1}s_{2}\dots s_{k-1} = r_{1}r_{2} \dots r_{k-1}$. Hence:

    $$ s_{1}s_{2} \dots s_{k-1}s_{k}s_{k-1} \dots s_{2}s_{1} = s_{1}s_{2} \dots s_{k-1}r_{k}s_{k-1} \dots s_{2}s_{1}$$

    $$ = r_{1}r_{2} \dots r_{k-1}r_{k}r_{k-1}\dots r_{2}r_{1}$$
    Hence, $ t := s_{1}s_{2} \dots s_{k-1}s_{k}s_{k-1}\dots s_{2}s_{1} =  r_{1}r_{2} \dots r_{k-1}r_{k}r_{k-1}\dots r_{2}r_{1}$ is well-defined.
\end{proof}

\begin{prop} \label{racsreflbijection}
    Let $ (W,S)$ be a right-angled Coxeter system, and let $ w\in W$ be any element (note that $ w$ is fully commutative since we are working in a right-angled Coxeter system). Then there is a bijection between the set of principle labeled order ideals of $ \H{w}$ and the set $ N(w)$ given by the following map:

    $$ I(q) \longmapsto t \in N(w) $$
    where $ t\in T$ is the reflection given in Proposition \ref{welldefreflection}.
\end{prop}

\begin{proof}
    By Proposition \ref{welldefreflection}, let $ x\weakr w$ such that $ \H{x} \cong I(q)$, and let $ \textbf{\underline{x}} = (s_{1},s_{2}, \dots , s_{k})$ be a reduced expression. Then $ t = s_{1}s_{2} \dots s_{k-1}s_{k}s_{k-1} \dots s_{2}s_{1}$. First, note that a quick calculation shows that $ tx = s_{1}s_{2} \dots s_{k-1}$, and hence $ \ell(tx) < \ell(x)$. Therefore, $ t\in N(x)$. But note that because $ x \weakr w$, Proposition \ref{weakrequivreflcocycle} implies that $ N(x) \subseteq N(w)$. Therefore, $ t\in N(w)$, and this shows that the codomain of the map $ I(q) \longmapsto t$ lies in $ N(w)$. We just need to show that the map $ I(q) \longmapsto t$ is a bijection between the set of principle labeled order ideals of $ \H{w}$ and the set $ N(w)$.

    Let $ q_{1},q_{2} \in \H{w}$. Let $ t_{1}, t_{2} \in N(w)$ be the reflections corresponding to $ I(q_{1})$ and $ I(q_{2})$ respectively. Let $ x_{1}, x_{2} \weakr w$ such that $ \H{x_{1}} \cong I(q_{1})$ and $ \H{x_{2}} \cong I(q_{2})$. Let $ \textbf{\underline{x}}_{1} = (r_{1},r_{2} , \dots , r_{m})$ and $ \textbf{\underline{x}}_{2} = (s_{1}, s_{2}, \dots , s_{n})$ be reduced expressions. Then:

    $$ t_{1} = r_{1}r_{2} \dots r_{m-1}r_{m}r_{m-1} \dots r_{2}r_{1}$$
    and

    $$ t_{2} = s_{1}s_{2} \dots s_{n-1}s_{n}s_{n-1} \dots s_{2}s_{1}$$
    I claim that the map $ I(q) \longmapsto t$ is injective. To prove this, suppose that $ t_{1} = t_{2}$. Note that $ t_{1} = t_{2}$ implies that the simply reflections $ r_{m}$ and $ s_{n}$ are conjugate. But by the exercise at the end of Section 5.3 of \cite{humphreys1992reflection}, two simple reflections are conjugate if and only if there exists a path of edges whose labels are odd integers $ \geq 3$ in the Coxeter diagram of $(W,S)$. Since $ (W,S)$ is a right-angled Coxeter system, no edges with odd labels $\geq 3$ occur in the Coxeter diagram of $(W,S)$. Hence, we deduce that $ r_{m} = s_{n}$. But note that $ r_{m}$ and $ s_{n}$ are the $S$-labels of $ q_{1}$ and $q_{2}$ respectively. Hence, $ q_{1}$ and $ q_{2}$ have the same $S$-label $ s: = r_{m} = s_{n}$. If $ q_{1} = q_{2}$, then $ I(q_{1}) = I(q_{2})$, and we conclude that the map is injective. Suppose for the sake of contradiction that $ q_{1} \neq q_{2}$. Note that because $ q_{1}$ and $ q_{2}$ have the same $S$-label $s$, it follows from Proposition \ref{mustbecompprop} that either $ q_{1} <_{w} q_{2}$ or $ q_{2} <_{w} q_{1}$. Without loss of generality, let us assume that $ q_{1} <_{w} q_{2}$. But because $ q_{1} <_{w} q_{2}$, we deduce that $ I(q_{1}) \subsetneq I(q_{2})$. Because $I(q_{1}) \subsetneq I(q_{2}) $, where $ \H{x_{1}} \cong I(q_{1})$ and $ \H{x_{2}} \cong I(q_{2})$, it follows from Proposition \ref{equividealprop} that $ x_{1} \sweakr x_{2}$. But since $ x_{1} \sweakr x_{2}$, we may assume that $ r_{i} = s_{i}$ for $ i=1,2, \dots, m$ and that $ m< n$. Hence, 

    $$ \textbf{\underline{x}}_{2}^{\ast} = (r_{1}, r_{2}, \dots , r_{m}, s_{m+1}, \dots, s_{n})$$
    is a reduced expression for $ x_{2}$ and 

    $$ t_{2} =  r_{1}r_{2} \dots  r_{m}s_{m+1} \dots s_{n-1}s_{n}s_{n-1} \dots s_{m+1}r_{m} \dots r_{2}r_{1}$$
    But note that because $ t_{1} = t_{2}$, we have that

    $$ x_{2} = t_{1}t_{2}x_{2} = t_{1}t_{2} r_{1}r_{2} \dots r_{m}s_{m+1} \dots s_{n}$$

    $$ =  t_{1} r_{1}r_{2} \dots r_{m}s_{m+1} \dots s_{n-1}$$

    $$ = r_{1}r_{2} \dots \widehat{r_{m}}s_{m+1} \dots s_{n-1}$$
    where the $\widehat{r_{m}}$ denotes omission of $ r_{m}$. But the conclusion $ x_{2} = r_{1}r_{2} \dots \widehat{r_{m}}s_{m+1} \dots s_{n-1}$ is a contradiction, since we have produced an expression for $ x_{2}$ whose length is strictly less than the length of the reduced expression $  \textbf{\underline{x}}_{2}^{\ast}$. Thus, we must have that $ I(q_{1}) = I(q_{2})$. This proves that the map $ I(q) \longmapsto t\in N(w)$ is injective. 

    To see that the injective map $ I(q) \longmapsto t\in N(w)$ is actually a bijection, note that the set of principle labeled order ideals of $ \H{w}$ has cardinality $ |\H{w}|$, and $ |\H{w}| = \ell(w)$. It is a general fact that $ |N(w)| = \ell(w)$. Thus, since the set of principle labeled order ideals of $ \H{w}$ has the same finite cardinality as the set $ N(w)$, and since the map $ I(q) \longmapsto t\in N(w)$ is an injection between these two sets, we conclude that the map is a bijection

\end{proof}

\begin{prop} \label{reflwordreduced}
    Let $ (W,S)$ be a right-angled Coxeter system, and let $ w\in W$. Let $ I(q) \subseteq \H{w}$ denote a principal labeled order ideal. Let $ t\in N(w)$ denote the reflection corresponding to $ I(q)$ by Proposition \ref{racsreflbijection}. Let 

    $$ t = s_{1}s_{2} \dots s_{k-1}s_{k}s_{k-1} \dots s_{2}s_{1}$$
    be as in Proposition \ref{welldefreflection} where $ x\weakr w$, $ \H{x} \cong I(q)$, and $ \textbf{\underline{x}} = (s_{1},s_{2}, \dots , s_{k})$ is some reduced expression. Then

    $$ \textbf{\underline{t}} = (s_{1},s_{2}, \dots, s_{k-1},s_{k},s_{k-1}, \dots ,s_{2},s_{1})$$
    is a reduced expression.
\end{prop}

\begin{proof}

    Consider $ \textbf{\underline{x}} = (s_{1},s_{2}, \dots , s_{k})$. Note that since $ \H{x} \cong I(q)$, $ s_{k}$ is the $S$-label of the unique maximal element $ q\in I(q)$. Thus, by Proposition \ref{removemaxprop}, it follows that $ D_{R}(x) = \{  s_{k} \}$. Since $ \textbf{\underline{x}} = (s_{1},s_{2}, \dots , s_{k})$ is reduced, applying any braid relation to $ \textbf{\underline{x}} = (s_{1},s_{2}, \dots , s_{k})$ still results in a reduced expression. But since we are working in a right-angled Coxeter system, the only braid relations are the commutation relations. Note that if $ \textbf{\underline{x}}^{\ast} = (r_{1}, r_{2} , \dots , r_{k})$ is \emph{any} reduced expression for $ x$, then $ r_{k} = s_{k}$, $ r_{k-1} \neq s_{k}$, and $ m(r_{k-1}, r_{k}) = \infty$. In particular, since $  m(r_{k-1}, r_{k}) = \infty$, it follows that one can never commute the $(k-1)$-th and $ k$-th entries of \emph{any} reduced expression for $ x$. Since $ \textbf{\underline{x}}^{-1} = (s_{k}, \dots , s_{2}, s_{1})$ is also a reduced expression, one can conclude using similar reasoning that one can never commute the first and second entries of \emph{any} reduced expression for $ x^{-1}$.

    Now suppose for the sake of contradiction that 

    $$ \textbf{\underline{t}} = (s_{1},s_{2}, \dots, s_{k-1},s_{k},s_{k-1}, \dots ,s_{2},s_{1})$$
    failed to be a reduced expression. Note that by Theorem 3.3.1 of \cite{bjorner2005combinatorics}, it follows that there exists a sequence of braid-moves on $\textbf{\underline{t}}$ that allows for a nil-move to occur. Since we are working in a right-angled Coxeter system, the only braid moves are commutation relations. But since $ \textbf{\underline{x}} = (s_{1},s_{2}, \dots , s_{k})$ occurs as a prefix of $ \textbf{\underline{t}}$, it follows that one can never commute the $ (k-1)$-th and $ k$-th entries of $ \textbf{\underline{t}}$. Similarly, since $ \textbf{\underline{x}}^{-1} = (s_{k}, \dots , s_{2}, s_{1})$ is a reduced expression that occurs as a postfix of $ \textbf{\underline{t}}$, it follows that one can never commute the $ k$-th and $(k+1)$-th entries of $ \textbf{\underline{t}}$. Hence, commutation-moves can only occur on the prefix $ (s_{1}, s_{2}, \dots , s_{k-1})$ or on the postfix $(s_{k-1}, \dots , s_{2}, s_{1})$ of $ \textbf{\underline{t}}$. But since there exists a sequence of commutation-moves on $\textbf{\underline{t}}$ that allows for a nil-move to occur, it follows that a nil-move will occur in $(s_{1}, s_{2}, \dots , s_{k-1}) $ or in $ (s_{k-1}, \dots , s_{2}, s_{1})$ after some sequence of commutation-moves. But this would imply that $(s_{1}, s_{2}, \dots , s_{k-1}) $ or $ (s_{k-1}, \dots , s_{2}, s_{1})$ fails to be a reduced expression, which would then imply that $ \textbf{\underline{x}}$ or $ \textbf{\underline{x}}^{-1}$ would fail to be a reduced expression, which is a contradiction. Hence, $ \textbf{\underline{t}}$ must be reduced.
\end{proof}

\begin{prop} \label{reflwordequal}
    Let $ (W,S)$ be a right-angled Coxeter system. Let $ t\in T $. Suppose that 

    $$ \textbf{\underline{t}}_{1} = (s_{1}, s_{2} , \dots s_{k-1}, s_{k}, s_{k-1} , \dots , s_{2}, s_{1}) $$
    and

    $$ \textbf{\underline{t}}_{2} = (r_{1}, r_{2} ,\dots , r_{k-1}, r_{k}, r_{k-1}, \dots , r_{2}, r_{1})$$
    are both palindromic reduced expressions for $ t$. Then $ s_{1}s_{2} \dots s_{k} = r_{1}r_{2} \dots r_{k}$.
\end{prop}

\begin{proof}
    Note that since $ \textbf{\underline{t}}_{1}$ and $ \textbf{\underline{t}}_{2}$ are palindromic expressions for the same element $t$, it follows that $ r_{k}$ and $ s_{k}$ are conjugate simple reflections. But in a right-angled Coxeter system, two simple reflections are conjugate if and only if they are equal. Hence, $ r_{k} = s_{k}$. More generally, in \emph{any} palindromic expression for $t$, the simple reflection in the exact middle entry must be equal to $ s: = r_{k} = s_{k}$. Since we are dealing with palindromic reduced expressions, one can use similar reasoning as in the proof of the previous proposition to conclude that a commutation-move can never occur between the $ (k-1)$-th and $ k$-th entries of $ \textbf{\underline{t}}_{1}$ and $ \textbf{\underline{t}}_{2}$. Similarly, a commutation-move can never occur between the $ k$-th and $(k+1)$-th entries of $ \textbf{\underline{t}}_{1}$ and $ \textbf{\underline{t}}_{2}$. But since $ \textbf{\underline{t}}_{1}$ and $ \textbf{\underline{t}}_{2}$ are both reduced expressions for $t$, it follows that there exists a sequence of commutation-moves that transforms $ \textbf{\underline{t}}_{1}$ to $ \textbf{\underline{t}}_{2}$. But when a commutation-move occurs on $ \textbf{\underline{t}}_{1}$, it must occur either on the prefix $ (s_{1}, s_{2} , \dots s_{k-1})$ or on the postfix $(s_{k-1} , \dots , s_{2}, s_{1}) $. Hence, $ (s_{1}, s_{2} , \dots s_{k-1})$ and $(s_{k-1} , \dots , s_{2}, s_{1}) $ can be transformed via commutation-moves to $(r_{1}, r_{2} ,\dots , r_{k-1}) $ and $(r_{k-1}, \dots , r_{2}, r_{1}) $ respectively. Thus, $ s_{1}s_{2} \dots s_{k-1} = r_{1}r_{2} \dots r_{k-1}$, and since $ r_{k} = s_{k}$, we deduce that

    $$ s_{1}s_{2} \dots s_{k} = r_{1}r_{2} \dots r_{k}$$
\end{proof}

\begin{prop} \label{gimenezthmonedirection}
    Let $ (W,S)$ be a right-angled Coxeter system. Let $ x,y \in W$. Let $ x\wedge y$ denote the meet of $ x$ and $ y$ in the weak right order of $ (W,S)$. Then $ N(x\wedge y) = N(x) \cap N(y)$.
\end{prop}

\begin{proof}
    By Proposition \ref{weakrequivreflcocycle}, we have that $ N(x \wedge y) \subseteq N(x)$ and $ N(x\wedge y) \subseteq N(y)$. Hence, $ N(x\wedge y) \subseteq N(x) \cap N(y)$. We just need to show the reverse inclusion.

    Let $ t\in N(x) \cap N(y)$. Since we are working in a right-angled Coxeter system, $ x$ and $ y$ are fully commutative. Thus, since $ t\in N(x)$ and $ t\in N(y)$, one can apply Proposition \ref{racsreflbijection} and Propostion \ref{welldefreflection} to construct principal labeled order ideals $ I(q_{1}) \subseteq \H{x}$ and $ I(q_{2}) \subseteq \H{y}$ such that $ t$ corresponds to $ I(q_{1})$ and $ I(q_{2})$. Let $ z_{1} \weakr x$ and $ z_{2} \weakr y$ be such that $ \H{z_{1}} \cong I(q_{1})$ and $ \H{z_{2}}\cong I(q_{2})$. Let 

    $$ \textbf{\underline{z}}_{1} = (s_{1},s_{2}, \dots, s_{m}) $$
    and

    $$ \textbf{\underline{z}}_{1} = (r_{1}, r_{2}, \dots, r_{n})$$
    denote reduced expressions. Then:

    $$ t= s_{1}s_{2} \dots s_{m-1}s_{m}s_{m-1} \dots s_{2}s_{1}$$
    and

    $$ t= r_{1}r_{2} \dots r_{n-1}r_{n}r_{n-1} \dots r_{2}r_{1}$$
    By Proposition \ref{reflwordreduced}, the expressions 

    $$ (s_{1}, s_{2} , \dots s_{m-1}, s_{m}, s_{m-1} , \dots , s_{2}, s_{1}) $$
    and

    $$ (r_{1}, r_{2} ,\dots , r_{n-1}, r_{n}, r_{n-1}, \dots , r_{2}, r_{1})$$
    are both reduced (and hence $ m=n$). By Proposition \ref{reflwordequal}, one deduces that $ z_{1} =  s_{1}s_{2}\dots s_{m} = r_{1}r_{2} \dots r_{n}  = z_{2}$. Let $ z: = z_{1} = z_{2}$. Since $ z \weakr x$ and $ z\weakr y$, we have that $ z \weakr x\wedge y$. Thus, $ N(z ) \subseteq N(x\wedge y)$. But note that $ t\in N(z) \subseteq N(x\wedge y)$. This shows that $ N(x)\cap N(y) \subseteq N(x\wedge y)$, and thus we have established that $ N(x\wedge y) = N(x) \cap N(y)$.
\end{proof}

\begin{proof}[Proof of Theorem \ref{gimenezracsthm}]
    Suppose $ (W,S)$ is a right-angled Coxeter system. Then Theorem \ref{dyerthm} and Proposition \ref{gimenezthmonedirection} imply that $ N(w) = \{ t\in T \mid \ell(tw) < \ell(w)  \}$ is the unique reflection cocycle associated to $ (W,S)$ and that $ N: W \rightarrow \mathcal{P}(T)$ satisfies the meet intersection condition.

    Suppose now that $ (W,S)$ is a reflection system with reflection cocycle $ N:W \rightarrow \mathcal{P}(T)$ that satisfies the meet intersectio condition. Note that Theorem \ref{dyerthm} implies that $ (W,S)$ is a Coxeter system. Let us suppose for the sake of contradiction that $ (W,S)$ failed to be a right-angle Coxeter system. This would imply that there exists some pair of simple reflections $ r,s\in S$ such that $ m(r,s) \in \N$ and $ m(r,s) \geq 3$. Consider the standard parabolic dihedral subsystem $ (W_{\{ r,s \} }, \{  r,s\})$. Note that the weak right order of a standard parabolic subsystem naturally embeds as an order ideal of the weak right order of the overall Coxeter system. Hence, if $ N: W\rightarrow \mathcal{P}(T)$ has the meet intersection property on all of $ (W,S)$, then $N$ should have the meet intersection property when restricted to $ (W_{\{ r,s \} } , \{ r,s \})$. Let $ m: = m(r,s)$. Since $ m\in \N$,  $ (W_{\{ r,s \} } , \{ r,s \})$ is a finite dihedral Coxeter system. Let $ w_{0}\in W_{\{ r,s \} }$ denote the longest element. Consider $ w_{0}r$ and $ w_{0}s$. Since $m\in \N$ and $m\geq 3$, it follows from a quick calculation that $ N(w_{0}r) \cap N(w_{0}s) = T^{\ast} \setminus \{ r,s \}$ where $ T^{\ast} : = \bigcup_{w\in W_{\{ r,s \} }} w\{ r,s \}w^{-1}$. But note that there cannot exist any $z\in W_{\{ r,s \} }$ such that $ N(z) = T^{\ast} \setminus \{ r,s \}$. This is because $N(z)$ must contain some simple reflection, but $ T^{\ast} \setminus \{  r,s \}$ does not contain any simple reflection. This is a contradiction, and therefore $ m(r,s) \in \{  2,\infty \}$ for all distinct $ r,s\in S$. This proves that $ (W,S)$ is a right-angled Coxeter system. 
    
\end{proof}

\bibliographystyle{abbrv}

\end{document}